\documentclass[11pt]{article}

\usepackage[utf8]{inputenc}
\usepackage{amsmath, amssymb, amsfonts, mathrsfs, bbm, bbold, array, booktabs, multirow}
\usepackage{graphicx}
\usepackage{caption}
\usepackage{geometry}
\usepackage{hyperref}
\usepackage{times}
\usepackage{float}
\usepackage{xcolor}
\usepackage{ragged2e}
\usepackage{csquotes}
\usepackage{longtable}
\usepackage{appendix}
\usepackage{amsthm}

\usepackage{amsthm}

\newtheorem{definition}{Definition}[section]
\newtheorem{lemma}[definition]{Lemma}
\newtheorem{proposition}[definition]{Proposition}
\newtheorem{theorem}[definition]{Theorem}
\newtheorem{corollary}[definition]{Corollary}
\newtheorem{remark}[definition]{Remark}

\title{On the Arc Length of a Supercircle and a Hypergeometric Formulation of $\pi$}

\author{
    Yomber Montilla\textsuperscript{1},
    R. Omar Rodriguez\textsuperscript{2},
    Brexys Linares\textsuperscript{3} \\
    \small \textsuperscript{1} Universidad T\'ecnica Estatal de Quevedo, Facultad de Ciencias de la Ingenier\'ia, Quevedo 120550, Ecuador \\
    \small \textsuperscript{2} Decanato de Ciencias y Tecnolog\'ia, Universidad Centroccidental Lisandro Alvarado, Barquisimeto 3001, Venezuela \\
    \small \textsuperscript{3} Universidad T\'ecnica de Manab\'i, Facultad de Posgrado, Portoviejo 130105, Ecuador \\
    \small Corresponding authors: \texttt{ymontillal@uteq.edu.ec}; \texttt{rafaelorodriguez@ucla.edu.ve}
}

\begin{document}
\date{}
\maketitle

\begin{abstract}
We obtain an infinite-series representation for the arc length of a supercircle in terms of the scale parameter $a$ and the shape parameter $n$. The resulting expression is constructed by means of generalized binomial coefficients and Gauss hypergeometric functions, distinguishing two regimes associated with the value of $n$. We also analyze the absolute convergence of the resulting series. We verify the consistency of the formulation from limiting cases and particular configurations of the family of supercircles: when $n\to0^+$ and $n\to\infty$, the length converges to the value $8a$, corresponding to the limiting rectilinear geometries, whereas for $n=1$ we recover the perimeter of the rhombus with diagonals of length $2a$. In addition, as a validation against supercircles with exact arc length, the formulation reproduces with high numerical precision the arc length of the parabolic star, the astroid, and the circle. Finally, by specializing the circular case $n=2$ and normalizing the length by the diameter $2a$, we obtain a series representation, in terms of hypergeometric functions, for the constant $\pi$.

\end{abstract}

\section{Introduction} 
Supercircles, also known as Lamé circles, form a family of closed plane curves that, by varying a shape parameter, describe a wide variety of geometries: concave four-cusped shapes, rhombi, circles, and near-square configurations. This geometric versatility has favored their use in describing natural forms, such as the geometry of certain plant structures~\cite{R-1,R-2,R-3,R-4}. In technology, they appear in the geometry of LCD screen pixels and in the corresponding far-field diffraction patterns~\cite{R-5}. Likewise, in architecture and design, they are used in ornamental objects and civil infrastructures, such as squares and stadiums~\cite{R-6}. 

Several works in physics and engineering have addressed the computation and characterization of basic geometric properties associated with the supercircle. Bera et al. (2008) present a compact formula, in terms of gamma functions, for the area enclosed by a Lam\'e circle, and they apply it to the computation of the ground-state energy of a particle confined in a box with supercircular geometry~\cite{R-7}. Likewise, Isoj\"arvi (2021) extends this study to the three-dimensional case by considering the dynamics of a particle trapped in a potential well with the geometry of a Lam\'e sphere of constant volume. For this purpose, he uses a parameterized formula for the $n$-dimensional volume of superspheres~\cite{R-8}. Matsuura (2014), in turn, determines the curvature and the asymptotic behavior of the maximum curvature of Lamé curves in order to solve the optimal design problem posed by Piet Hein for Sergel Square in Stockholm~\cite{R-9}. These studies highlight the importance of considering the geometric properties of supercircles, since many models are based on rectangles and circles, whereas supercircles make it possible to approximate certain real geometries more appropriately. Regarding arc length, Erbaş (2022) points out that, although the computation of the area of supercircles is analytically accessible, there is a notable gap concerning closed formulas for their arc length. In this context, he proposes a perimeter expression evaluated by numerical integration for the curve under consideration, and applies it to boundary-value problems in rectangular domains modeled by the Laplace equation~\cite{R-10}. 

For some particular cases of the family of supercircles, such as the astroid, the rhombus, the circle, and the limiting square case, the arc length admits closed expressions. However, for other curves of interest in engineering, such as the \textit{squircle}, the value of the shape parameter leads to integrals without a closed-form solution, which makes it difficult to obtain exact expressions for their arc length. This motivates the present work, which aims to derive a compact expression for the arc length of a supercircle. The geometric richness of this family of curves makes it a suitable object for studying how the arc length varies under deformations controlled by the shape parameter. 

The central contribution of this work is to obtain a formulation for the arc length of a supercircle, depending on the shape parameter $n$ and the scale parameter $a$, by means of a series of Gauss hypergeometric functions weighted by generalized binomial coefficients. Due to the axial and diagonal symmetries of the curve, the total length is reduced to the study of the arc corresponding to the angular interval $\theta \in [0,\pi/4]$. In this domain, the classical arc-length formula in polar coordinates, together with suitable changes of variables, allows the original integral to be transformed into a form in which generalized binomial coefficients and Gauss hypergeometric functions arise naturally. The resulting expression is determined by two regimes associated with the shape parameter $n$. As a consistency check, the analysis of the limits $n\to 0^+$ and $n\to\infty$ shows that the arc length converges at both endpoints to the value $8a$. In particular, as $n\to\infty$, the supercircle converges geometrically to the square of side $2a$, and the formulation reproduces its perimeter. Likewise, at the transition value $n=1$, the curve reduces to a rhombus and the calculation recovers its length $4\sqrt{2}\,a$, which constitutes the lower bound of the family. Finally, we obtain a hypergeometric representation of $\pi$ by specializing the formulation to the circular case, in agreement with the fact that series representations of this constant continue to appear in various contexts of mathematical analysis and mathematical physics, including recent formulations derived from amplitude expansions in string theory~\cite{R-11}. 

Since the formulation obtained for the arc length is expressed as an infinite series depending on $n$, its evaluation is carried out by successive truncations. The numerical analysis makes it possible to study the global behavior of $L(n)$ and to validate the expression in particular cases with exact lengths. The results show that the required truncation order is not uniform over the parameter domain, since it increases significantly near $n=1$ and decreases as one moves away from this value. In addition, the evaluation of classical cases---the parabolic star, the astroid, the rhombus, and the circle---confirms that the formulation reproduces the corresponding analytical lengths with high precision. Finally, for the hypergeometric representation of $\pi$, we observe that, when the prescribed tolerance is decreased by one order of magnitude, approximately one additional correct digit is obtained. 

The article is organized as follows. In Section 2 we derive the integral representation of the arc length of the supercircle and establish the regimes determined by the parameter $n$. In Section 3 we obtain the formulation in terms of hypergeometric functions and analyze the convergence of the corresponding series. We then study the rectilinear geometries of the family and the circular case, from which we derive a hypergeometric representation of $\pi$. Finally, Section 4 presents the numerical results that validate the formulation obtained, and the last section summarizes the main conclusions.

\section{Arc Length of a Supercircle}

\begin{definition}
Let $a>0$ and let $n>0$. The supercircle with scale parameter $a$ and shape parameter $n$ is defined in polar coordinates by~\cite{R-7}
\begin{equation}\label{Eq1}
    r(\theta) = a\left(|\cos\theta|^n + |\sin\theta|^n\right)^{-1/n}\ ,
\end{equation}
where $a$ is a scale parameter, whereas $n$ controls the degree of squareness.
\end{definition}

\begin{figure}[h]
\centering
\includegraphics[width=0.50\linewidth]{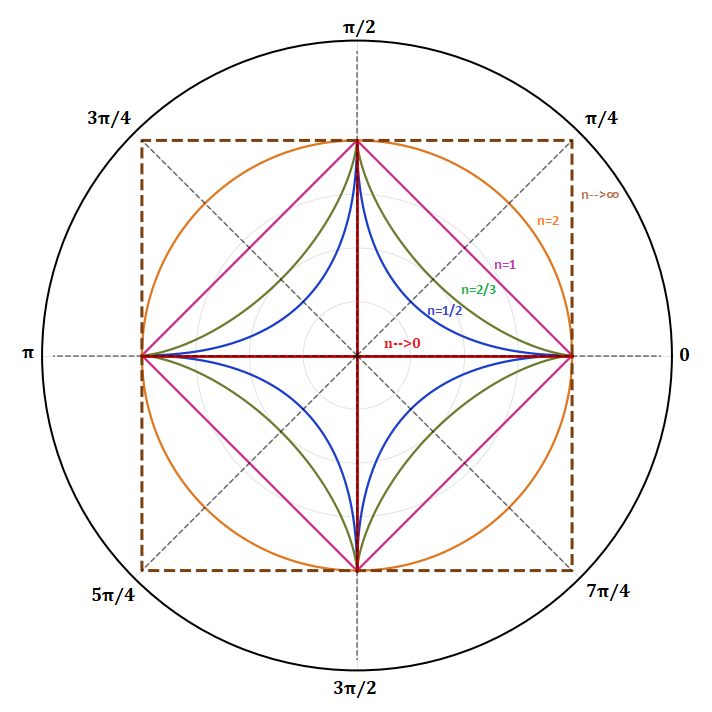}
\caption{Family of supercircles with $a=1$: $n\to 0$ (cross), $n=1/2$ (parabolic arcs),
$n=2/3$ (astroid), $n=1$ (rhombus), $n=2$ (circle), and $n\to\infty$
(square). The dashed lines illustrate the angular partition into
octants of width $\pi/4$.}
\label{F1}
\end{figure}

Fig.~\ref{F1} illustrates some particular cases of supercircles. The curve has axial and diagonal symmetry. By these symmetries, the total length can be obtained from the arc contained in the angular interval $\theta\in[0,\pi/4]$. In this domain, $0 \leq \cos\theta, \sin\theta \leq 1$, so Eq.~\eqref{Eq1} can be written without absolute values as
\begin{equation}\label{Eq2}
    r(\theta) = a\left(\cos^n\theta + \sin^n\theta\right)^{-1/n}, \quad \theta \in [0, \pi/4].
\end{equation}

\begin{lemma}
For $a>0$ and $n>0$, the arc length of the supercircle defined by Eq.~\eqref{Eq1} can be reduced to the interval $\theta\in[0,\pi/4]$ by
\begin{equation}\label{Eq3}
    L = 8\int_0^{\pi/4} \sqrt{r^2 + \left(\frac{dr}{d\theta}\right)^2}\, d\theta.
\end{equation}
\end{lemma}

\begin{proof}
The curve defined by Eq.~\eqref{Eq1} is invariant under reflections with respect to the coordinate axes and under interchange of the coordinate variables. In polar coordinates, these symmetries allow the complete curve to be decomposed into eight congruent arcs, each corresponding to an angular interval of width $\pi/4$. Therefore, the total length is obtained by multiplying the length of the arc in $\theta\in[0,\pi/4]$ by $8$. In this interval, the classical arc-length formula in polar coordinates leads directly to Eq.~\eqref{Eq3}.
\end{proof}

\begin{proposition}
The arc length of the supercircle admits the integral representation
\begin{equation}\label{Eq14}
    L = \frac{8a}{n}\int_0^{1} \xi^{-(n-1)/n}\left(1+\xi\right)^{-(1+1/n)}
    \left[1+\xi^{2(n-1)/n}\right]^{1/2}\;d\xi.
\end{equation}
\end{proposition}

\begin{proof}
The direct calculation of the integral in Eq.~\eqref{Eq3} is not trivial. In particular, the power $-1/n$ appearing in Eq.~\eqref{Eq2} makes the algebraic manipulation of the quadratic terms under the radical difficult. For this reason, we introduce an auxiliary variable that allows the integrand to be expressed in a more analytically tractable form. By defining
\begin{equation}\label{Eq4}
    z(\theta) = \cos^n\theta + \sin^n\theta,
\end{equation}
we obtain
\begin{equation}\label{Eq5}
    r(\theta) = az^{-1/n}.
\end{equation}
The derivative of $r$ with respect to $\theta$, obtained by the chain rule, is given by
\begin{equation}\label{Eq6}
        \frac{dr}{d\theta} 
    = \frac{dr}{dz}\frac{dz}{d\theta} 
    = -\frac{a}{n}\,z^{-1/n}\,\frac{z'}{z}.
\end{equation}
Substituting Eqs.~\eqref{Eq5} and~\eqref{Eq6} into Eq.~\eqref{Eq3}, we obtain
\begin{equation}\label{Eq7}
    L = 8\int_0^{\pi/4} \left[a^2z^{-2/n} +\frac{a^2z^{-2/n}}{n^{2}} 
    \left(\frac{z^{\prime}}{z}\right)^2\right]^{1/2} d\theta
    = \frac{8a}{n}\int_0^{\pi/4}z^{-(1/n+1)} 
    \left[n^{2}z^{2} +(z^{\prime})^{2}\right]^{1/2} d\theta.
\end{equation}
Since the integrand is expressed in terms of $z$ and $z'$, the quadratic terms contained in the radical can be expanded and simplified without ambiguity in the integration limits.

Expanding the term $n^2z^2$ by using Eq.~\eqref{Eq4}, we obtain:
\begin{align}\label{Eq8}
    n^{2}z^2 &= n^{2}\cos^{2n}\theta + n^{2}\sin^{2n}\theta + 2n^{2}(\sin\theta\cos\theta)^n \nonumber\\
    &= n^{2}\cos^{2(n-1)}\theta\cos^{2}\theta + n^{2}\sin^{2(n-1)}\theta\sin^{2}\theta
    +2n^{2}(\sin\theta\cos\theta)^n.
\end{align}
In turn, the derivative of $z$ is given by
\begin{equation}\label{Eq9}
    z^{\prime} = n\sin^{n-1}\theta\cos\theta - n\cos^{n-1}\theta\sin\theta,
\end{equation}
and its square is
\begin{equation}\label{Eq10}
    (z^{\prime})^2 = n^{2}\sin^{2(n-1)}\theta\cos^{2}\theta 
    + n^{2}\cos^{2(n-1)}\theta\sin^{2}\theta
    - 2n^{2}(\sin\theta\cos\theta)^n.
\end{equation}
Adding Eqs.~\eqref{Eq8} and~\eqref{Eq10}, the cross terms $\pm 2n^{2}(\sin\theta\cos\theta)^{n}$ cancel, and, using the trigonometric identity $\cos^{2}\theta + \sin^{2}\theta = 1$, we obtain
\begin{equation}\label{Eq11}
    n^2z^2 + (z^{\prime})^2 =n^2\left[\cos^{2(n-1)}\theta + \sin^{2(n-1)}\theta\right].
\end{equation}
Substituting Eqs.~\eqref{Eq11} and~\eqref{Eq4} 
into Eq.~\eqref{Eq7}, and after factoring the powers of $\cos\theta$, the arc length is rewritten as
\begin{align}\label{Eq12}
    L=8a\int_0^{\pi/4} \left(1+\tan^n\theta\right)^{-(1+1/n)}
\left[1+\tan^{2(n-1)}\theta\right]^{1/2}\sec^2\theta\;d\theta.
\end{align}

Note that in Eq.~\eqref{Eq12} the factor $\sec^2\theta$ 
appears naturally together with $\tan^n\theta$, which suggests introducing the change of variable
\begin{equation}\label{Eq13}
    \xi = \tan^n\theta \implies \tan\theta = \xi^{1/n}.
\end{equation}
Since $\theta\in[0,\pi/4]$, this change transforms the integration interval into $\xi\in[0,1]$. Thus, the arc length is expressed as indicated in Eq.~\eqref{Eq14}.
\end{proof}

Eq.~\eqref{Eq14} can be sensitive to singularities depending on the value of $n$ when the integrand is evaluated in a neighborhood of $\xi\to 0^+$. This leads us to distinguish two regimes.

\begin{proposition}
The integral representation of the arc length can be written in the following two regimes:
\begin{equation}\label{Eq16}
        L = \frac{8a}{n}\int_0^1 
        \left(1+\xi\right)^{-(1+1/n)}
        \left(1+\xi^{2(1-n)/n}\right)^{1/2}\, d\xi,
        \qquad 0 < n \leq 1,
\end{equation}
and
 \begin{equation}\label{Eq20}
        L = \frac{8a}{n}\int_0^1 
        \xi^{-(n-1)/n}
        \left(1+\xi\right)^{-(1+1/n)}
        \left[1+\xi^{2(n-1)/n}\right]^{1/2} d\xi,
        \qquad n \geq 1.
    \end{equation}
\end{proposition}

\begin{proof}
We first consider the case $0<n<1$. In this interval, $2(n-1)/n<0$, so that $\xi^{2(n-1)/n}\to\infty$ as $\xi\to 0^+$. This singularity at the lower limit suggests rewriting the radical in Eq.~\eqref{Eq14} as
    \begin{equation}\label{Eq15}
        \left(1+\xi^{2(n-1)/n}\right)^{1/2}
        =
        \xi^{(n-1)/n}
        \left(1+\xi^{2(1-n)/n}\right)^{1/2}.
    \end{equation}
In this way, Eq.~\eqref{Eq14} reduces to the expression given in Eq.~\eqref{Eq16}, valid for $0<n\leq1$.

For $n=1$, the integral in Eq.~\eqref{Eq14} has no singularity at the lower endpoint, since the integrand reduces to $\sqrt{2}\,(1+\xi)^{-2}$. However, for $n>1$ the integral is improper at the lower endpoint, because the dominant singular factor $\xi^{-(n-1)/n}\to\infty$ as $\xi\to0^+$, whereas the remaining factors are bounded. To evaluate the convergence of the integral in Eq.~\eqref{Eq14}, we note that, for $0<\xi\leq1$ and $n>1$,
\begin{equation}\label{Eq17}
    0<(1+\xi)^{-(1+1/n)}\leq 1,
    \qquad
    1\leq \left(1+\xi^{2(n-1)/n}\right)^{1/2}\leq \sqrt{2}.
\end{equation}
Therefore, by combining the inequalities and multiplying by the positive factor $\xi^{-(n-1)/n}$,
\begin{equation}\label{Eq18}
    0\leq
    \xi^{-(n-1)/n}(1+\xi)^{-(1+1/n)}
    \left(1+\xi^{2(n-1)/n}\right)^{1/2}
    \leq
    \sqrt{2}\,\xi^{-(n-1)/n}.
\end{equation}
Since $-(n-1)/n>-1$ for all $n>1$, we have
\begin{equation}\label{Eq19}
     \sqrt{2}\int_0^1 \xi^{-(n-1)/n}\,d\xi = \sqrt{2} n.
\end{equation}
Consequently, by the direct comparison test, the integral in Eq.~\eqref{Eq14} converges for all $n>1$. Hence, for $n\geq1$ the arc length is expressed as indicated in Eq.~\eqref{Eq20}.
\end{proof}

\begin{remark}
For $n=1$, Eqs.~\eqref{Eq16} and~\eqref{Eq20} coincide. It is worth noting that, in general, these integrals do not admit representations in terms of elementary functions. Therefore, in the next section we develop their representation by means of special functions.
\end{remark}

\section{Arc Length in Terms of Hypergeometric Functions}

\begin{theorem}\label{T3.1}
Let $a>0$ and let $n>0$. The arc length of the supercircle defined in~\eqref{Eq1} is given by
\begin{equation}\label{Eq26}
    L = \frac{8a}{n}\sum_{m=0}^\infty \frac{1}{u_m}\binom{1/2}{m}
        \,{}_2F_1\!\left(1+\frac{1}{n},\,u_m;\,u_m+1;\,-1\right),
\end{equation}
where
\begin{equation}\label{Eq23}
u_m =
    \begin{cases}
        \displaystyle 1 + \frac{2m(1-n)}{n}, & 0 < n \leq 1,\\[10pt]
        \displaystyle \frac{1+2m(n-1)}{n},   & n \geq 1.
    \end{cases}
\end{equation}
\end{theorem}

\begin{proof}
The radical factor appearing in Eqs.~\eqref{Eq16} and~\eqref{Eq20} has the form $(1+\xi^{\lambda})^{1/2}$, with $\lambda\geq 0$ in each of the regimes considered. Since $\alpha=1/2>0$ and $\xi^\lambda\in[0,1]$ on the integration interval, we have
\begin{equation}\label{Eq21}
(1+\xi^{\lambda})^{1/2}
=
\sum_{m=0}^\infty \binom{1/2}{m} \xi^{\lambda m}.
\end{equation}
Eq.~\eqref{Eq21} is the binomial series with parameter $\alpha=1/2$ \cite{R-13,R-14}; see Appendix~\ref{AA}. Therefore, Eqs.~\eqref{Eq16} and~\eqref{Eq20}, with $\lambda=2(1-n)/n$ for $0 < n \leq 1$ and $\lambda=2(n-1)/n$ for $ n \geq 1 $, respectively, can be written compactly as
\begin{equation}\label{Eq22}
    L= \frac{8a}{n}\sum_{m=0}^\infty \binom{1/2}{m}
        \int_0^1 \xi^{u_m-1}(1+\xi)^{-(1+1/n)}\,d\xi, 
\end{equation}
where $u_m$ is defined in Eq.~\eqref{Eq23}.
 
In Eq.~\eqref{Eq22}, we recognize Euler's integral representation for the Gauss hypergeometric function \cite{R-15}:
\begin{equation}\label{Eq24}
    {}_2F_1(v,u;w;\gamma)
    =
    \frac{\Gamma(w)}{\Gamma(u)\,\Gamma(w-u)}
    \int_0^1 
    \xi^{u-1}
    (1-\xi)^{w-u-1}
    (1-\gamma\xi)^{-v}\,d\xi.
\end{equation}
For this choice of parameters, the factor $(1-\xi)^{w-u-1}$ reduces to unity by setting $w=u+1$, whereas $(1-\gamma \xi)^{-v}$ coincides with the factor $(1+\xi)^{-(1+1/n)}$ by taking $\gamma=-1$ and $v=1+1/n$. Consequently,
\begin{equation}\label{Eq25}
    \frac{1}{u}\,{}_2F_1\!\left(1+\frac{1}{n},u;u+1;-1\right)
    =\int_0^1 
    \xi^{u-1}
    (1+\xi)^{-(1+1/n)}\,d\xi,
\end{equation}
since $\Gamma(1)=1$ and $\Gamma(u+1)=u\,\Gamma(u)$. Applying the identity in Eq.~\eqref{Eq25} to the integral appearing in Eq.~\eqref{Eq22}, we obtain precisely the hypergeometric representation of the arc length given in Eq.~\eqref{Eq26}. This compact expression unifies both regimes of the family through the parameter $u_m$, whose form distinguishes the case $0<n\leq 1$ from the case $n\geq 1$.
\end{proof}
 
\subsection{Convergence of the Series}

\begin{proposition}
The hypergeometric series defining the arc length in Eq.~\eqref{Eq26} converges absolutely for all $n>0$.
\end{proposition}

\begin{proof}
Once we obtain the hypergeometric representation of the arc length in Eq.~\eqref{Eq26}, it is necessary to establish the convergence of the series that defines it. To this end, we first estimate the hypergeometric factor through its integral representation and then combine this bound with the asymptotic behavior of the generalized binomial coefficient. The analysis distinguishes the case $n\neq 1$, where the parameter $u_m$ grows linearly with $m$, from the particular case $n=1$, in which $u_m=1$ for all $m$.
 
Since $n>0$, we have $1+1/n>0$. Moreover, if $0<\xi\leq 1$, then $1<1+\xi\leq 2$. Therefore, $(1+\xi)^{1+1/n}\geq 1$ and, by taking reciprocals, we obtain a positive factor that does not exceed unity:
\begin{equation}\label{Eq27}
0 < (1 + \xi)^{-(1+1/n)} \leq 1.
\end{equation}
 
The positivity of $u_m$ implies that, for all $0<\xi\leq 1$, we have $\xi^{u_m-1}>0$. Thus, by multiplying the inequality~\eqref{Eq27} by the positive factor $\xi^{u_m-1}$, the direction of the inequality is preserved. Hence, we obtain
\begin{equation}\label{Eq28}
0 < \xi^{u_m-1}(1 + \xi)^{-(1+1/n)} \leq \xi^{u_m-1}.
\end{equation}
Integrating the inequality~\eqref{Eq28} over the interval $(0,1]$ and considering, when necessary, the improper integral at $\xi=0$, we obtain
\begin{equation}\label{Eq29}
0 < \int_0^1 \xi^{u_m-1}(1 + \xi)^{-(1+1/n)}\,d\xi
\leq \int_0^1 \xi^{u_m-1}\,d\xi.
\end{equation}
The condition $u_m>0$ guarantees that the integral on the right-hand side of~\eqref{Eq29} converges and is equal to $1/u_m$. Therefore,
\begin{equation}\label{Eq30}
0 < \int_0^1 \xi^{u_m-1}(1 + \xi)^{-(1+1/n)}\,d\xi \leq \frac{1}{u_m}.
\end{equation}
For $n>0$, $n\neq 1$, the parameter $u_m$ grows linearly in $m$. Indeed, $u_m$ can be written in the form
\begin{equation}\label{Eq31}
u_m = u_{0}(n) + \frac{2|1-n|}{n}\,m,
\qquad
u_0(n) =
\begin{cases}
1, & 0<n<1, \\[4pt]
\dfrac{1}{n}, & n>1.
\end{cases}
\end{equation}
Therefore, for all $m\geq 1$ and $n\neq 1$,
\begin{equation}\label{Eq32}
u_m \geq \frac{2|1-n|}{n}\,m.
\end{equation}
Since both sides of the inequality~\eqref{Eq32} are positive, taking reciprocals reverses the direction of the inequality. Consequently,
\begin{equation}\label{Eq33}
\frac{1}{u_m} \leq \frac{n}{2|1-n|}\frac{1}{m},
\qquad m\geq 1,\ n\neq 1.
\end{equation}
Combining the estimate~\eqref{Eq33} with the inequality~\eqref{Eq30}, we obtain
\begin{equation}\label{Eq34}
0 < \int_0^1 \xi^{u_m-1}(1+\xi)^{-(1+1/n)}\,d\xi
\leq \frac{n}{2|1-n|}\frac{1}{m},
\qquad m\geq 1,\ n\neq 1.
\end{equation}
Therefore, from the integral identity established in Eq.~\eqref{Eq25},
\begin{equation}\label{Eq35}
0 < \frac{1}{u_m}{}_2F_1\left(1+\frac{1}{n},u_m;u_m+1;-1\right)
\leq \frac{n}{2|1-n|}\frac{1}{m},
\qquad m\geq 1,\ n\neq 1.
\end{equation}
Consequently, the bound~\eqref{Eq35} implies that, for each fixed $n>0$, with $n\neq 1$ and for large values of $m$,
\begin{equation}\label{Eq36}
\frac{1}{u_m}{}_2F_1\left(1+\frac{1}{n},u_m;u_m+1;-1\right)
\sim O\left(m^{-1}\right). \qquad 
\end{equation}
On the other hand, the generalized binomial coefficient satisfies, as established in Appendix~\ref{AA}, the asymptotic behavior
\begin{equation}\label{Eq37}
\left|\binom{1/2}{m}\right| \sim O\left(m^{-3/2}\right).
\end{equation}
Combining this estimate with Eq.~\eqref{Eq36}, we obtain
\begin{equation}\label{Eq38}
\left|\frac{1}{u_m}\binom{1/2}{m} \, {}_2F_1\left(1+\frac{1}{n},\,u_m;\,u_m+1;\,-1\right)\right| \sim O\left(m^{-5/2}\right).
\end{equation}
Since $\displaystyle\sum_{m=1}^{\infty} m^{-5/2}=
\zeta(5/2)$ converges, the comparison test implies that the series~\eqref{Eq26} converges absolutely for each fixed $n>0$, with $n\neq 1$.
 
In the case $n=1$, we have $u_m=1$ and, according to Appendix~\ref{CC}, ${}_2F_1\!\left(2,1;2;-1\right)=1/2$. Therefore,
\begin{equation}\label{Eq39}
\left|\binom{1/2}{m}\,{}_2F_1\!\left(2,1;2;-1\right)\right|\sim O\!\left(m^{-3/2}\right).
\end{equation}
Since $\displaystyle\sum_{m=1}^{\infty} m^{-3/2}=\zeta(3/2)$ converges, the comparison test implies that the series~\eqref{Eq26} converges absolutely for $n=1$. Consequently, the series~\eqref{Eq26} converges absolutely for all $n>0$.
\end{proof}

\begin{remark}
The case $n=1$ represents the slowest convergence regime within the estimates obtained, since the general term is controlled by the asymptotic bound associated with the generalized binomial coefficient, of order $O(m^{-3/2})$, in contrast to the bound $O_n(m^{-5/2})$ obtained for $n\neq 1$. Moreover, since the integral factor defined in~\eqref{Eq25} is positive for all $m\geq 0$ and $n>0$, the sign of the general term is determined by $\binom{1/2}{m}$, so that the series has an alternating sign pattern from $m=1$ onward. The numerical results presented below are consistent with this convergence behavior.
\end{remark}

\subsection{Rectilinear Geometries of the Family}

\begin{proposition}
In the rectilinear cases of the family of supercircles, the following lengths are obtained:
\[
L_{\square}=8a,\qquad
L_{+}=8a,\qquad
L_{\diamond}=4\sqrt{2}\,a.
\]
\end{proposition}

\begin{proof}
It is of particular interest to evaluate Eq.~\eqref{Eq26} at those values of $n$ for which the supercircle reduces, or converges, to rectilinear configurations whose lengths are accessible through elementary geometric arguments. These cases, identified in Fig.~\ref{F1}, correspond to the limits $n\to 0^+$, $n\to\infty$, and the transition value $n=1$.
 
In the limit $n\to\infty$, the supercircle converges to the square 
of side $2a$. As $n$ decreases to $1$, the curve deforms 
continuously into a rhombus with diagonals of length $2a$. This value separates the two geometric behaviors 
of the family: for $n>1$ the curve is convex, whereas for 
$0<n<1$ it has concavity oriented toward the origin. Finally, when $n\to 0^+$, the curve contracts toward the coordinate axes and adopts a degenerate cross-shaped configuration, whose limiting path consists, by symmetry, of four segments of length $2a$.
 
The cases $n\to 0^+$, $n=1$, and $n\to\infty$ are the only members of the family composed exclusively of rectilinear segments. We now evaluate Eq.~\eqref{Eq26} for each of them.
 
\begin{itemize}
  \item \textbf{Square: $\mathbf{n\to\infty}$}.
 This case corresponds to the regime $n\geq 1$, for which $u_m$ is given by
Eq.~\eqref{Eq23}. For the term $m=0$ we have $u_0 = 1/n$, whereas, for all $m\geq 1$, $u_m \simeq 2m$ when $n$ is very large. Since the factor $1/u_0 = n$ diverges
as $n\to\infty$, we extract the term $m=0$ before taking
the limit, and rewrite Eq.~\eqref{Eq26} as
  \begin{equation}\label{Eq40}
      L = 8a\,\Bigg[\underbrace{{}_2F_1\!\left(1+\frac{1}{n},\,\frac{1}{n};\,
      \frac{1}{n}+1;\,-1\right)}_{S_0}
      +
      \frac{1}{n}
      \underbrace{\sum_{m=1}^\infty \frac{1}{u_m}\binom{1/2}{m}
      \,{}_2F_1\!\left(1+\frac{1}{n},\,u_m;\,u_m+1;\,-1\right)}_{S_1}
      \Bigg].
  \end{equation}
  \textbf{Limit of $\mathbf{S_0}$.}
  Since the first and third parameters of ${}_2F_1$ coincide in $S_0$,
  we apply the identity \cite{R-16}
  \begin{equation}\label{Eq41}
      {}_2F_1(v,\,u;\,v;\,\gamma) = (1-\gamma)^{-u},
  \end{equation}
  with $v=1+1/n$, $u=1/n$, and $\gamma=-1$, and obtain
  \begin{equation}\label{Eq42}
      S_0
      = {}_2F_1\!\left(1+\frac{1}{n},\,\frac{1}{n};\,
      1+\frac{1}{n};\,-1\right)
      = 2^{-1/n}.
  \end{equation}
  In the limit $n\to\infty$, we have $S_0\to 1$.
 
  \textbf{Limit of $\mathbf{S_1}$.}
 When $n\to\infty$, for each fixed $m\geq 1$, we have $u_m\to 2m$, so that
  \begin{equation}\label{Eq43}
      S_1 \;\xrightarrow{n\to\infty}\;
      \sum_{m=1}^\infty \frac{1}{2m}\binom{1/2}{m}
      \,{}_2F_1\!\left(1,\,2m;\,2m+1;\,-1\right).
  \end{equation}
  This series converges to the finite value $\ln 4 - \sqrt{2}\ln(1+\sqrt{2})$; see Appendix~\ref{BB}.
Therefore, by taking $n\to\infty$, $S_1/n\to 0$, and, consequently,
Eq.~\eqref{Eq40} yields
  \begin{equation}\label{Eq44}
      L_\square = 8a,
  \end{equation}
  a result consistent with the perimeter of the square of side $2a$.
 
  \item \textbf{Cross: $\mathbf{n\to 0^+.}$}
  This case corresponds to the regime $0 < n \leq 1$, for which $u_m$ is given by
Eq.~\eqref{Eq23}. For the term $m=0$ we have $u_0=1$, whereas for all $m\geq 1$ we have $u_m\simeq 2m/n$ when $n$ is positive and sufficiently close to zero. For the terms $m\geq 1$, note that
\begin{equation}\label{Eq45}
      \frac{1}{n}\cdot\frac{1}{u_m}
      \simeq\frac{1}{2m},
  \end{equation}
  does not diverge, since the factors of $n$ cancel. However, for $m=0$ we have 
  \begin{equation}\label{Eq46}
  \frac{1}{n}\cdot\frac{1}{u_0}
  =\frac{1}{n}\xrightarrow{n\to 0^+}\infty,
  \end{equation}
  so we extract the term $m=0$
  before taking the limit, and rewrite
  Eq.~\eqref{Eq26} as
  \begin{equation}\label{Eq47}
      L = 8a\,\Bigg[\frac{1}{n}
      \underbrace{{}_2F_1\!\left(1+\frac{1}{n},\,1;\,
      2;\,-1\right)}_{S_0}
      +
      \frac{1}{n}
      \underbrace{\sum_{m=1}^\infty \frac{1}{u_m}\binom{1/2}{m}
      \,{}_2F_1\!\left(1+\frac{1}{n},\,u_m;\,u_m+1;\,-1\right)}_{S_1}
      \Bigg].
  \end{equation}
 
  \textbf{Limit of $\mathbf{S_0}$}.
  Since $S_0$ reduces to the elementary form
  $n(1-2^{-1/n})$, see Appendix~\ref{CC}, we obtain
  \begin{equation}\label{Eq48}
      \frac{1}{n}\,S_0
      = \frac{1}{n}\,{}_2F_1\!\left(1+\frac{1}{n},\,1;\,2;\,-1\right)
      = 1-2^{-1/n},
  \end{equation}
  so that, by taking the limit $n\to 0^+$, we have $S_0/n\;\xrightarrow\; 1$.
  
  \textbf{Limit of $\mathbf{S_1}$}.
  When $n\to 0^+$, we have $u_m\simeq 2m/n\to\infty$,
  so that the second and third parameters of ${}_2F_1$
  are asymptotically equivalent: $u_m\simeq u_m+1\simeq 2m/n$.
  Applying the symmetry of the first two parameters of ${}_2F_1$ together with the identity~\eqref{Eq41}, we obtain
  \begin{equation}\label{Eq49}
      {}_2F_1\!\left(\frac{1}{n},\,\frac{2m}{n};\,
      \frac{2m}{n};\,-1\right) = 2^{-1/n}.
  \end{equation}
  Therefore, using the approximation~\eqref{Eq45}, 
  we have
  \begin{equation}\label{Eq50}
      \frac{1}{n}S_1
      \;\xrightarrow{n\to 0^+}\;
      2^{-1/n}\sum_{m=1}^\infty \frac{1}{2m}\binom{1/2}{m}
      = 2^{-1/n}\!\left(\sqrt{2}-1
        -\ln\!\left(\frac{1+\sqrt{2}}{2}\right)\right).
  \end{equation}
  Since $2^{-1/n}\to 0$ as $n\to 0^+$ and the sum is
  finite, we conclude that $S_1/n\to 0$,
  and Eq.~\eqref{Eq47} yields
  \begin{equation}\label{Eq51}
      L_+ = 8a,
  \end{equation}
  a result consistent with the perimeter of the limiting cross-shaped contour, composed of eight segments of length $a$.
 
  \item \textbf{Rhombus: $\mathbf{n=1}$}.
  For this parameter value, $u_m=1$ for all $m\geq 0$, so that
  Eq.~\eqref{Eq26} reduces to
  \begin{equation}\label{Eq52}
      L = 8a\sum_{m=0}^\infty \binom{1/2}{m}
          \,{}_2F_1\!\left(2,\,1;\,2;\,-1\right).
  \end{equation}
 Applying the identity~\eqref{Eq41} with $v=2$, $u=1$,
  and $\gamma=-1$, we obtain
  \begin{equation}\label{Eq53}
      {}_2F_1(2,1;2;-1) = (1+1)^{-1} = \frac{1}{2}.
  \end{equation}
  Substituting this result into Eq.~\eqref{Eq52}, we obtain
  \begin{equation}\label{Eq54}
      L = 4a\sum_{m=0}^\infty \binom{1/2}{m}.
  \end{equation}
  The sum corresponds to the generalized binomial series
evaluated at the point $x=1$:
  \begin{equation}\label{Eq55}
      \sum_{m=0}^\infty \binom{1/2}{m} = (1+1)^{1/2} = \sqrt{2},
  \end{equation}
  hence
  \begin{equation}\label{Eq56}
      L_\diamond = 4\sqrt{2}\,a,
  \end{equation}
  a result consistent with the perimeter of the rhombus with diagonals of length $2a$,
whose four sides have length $a\sqrt{2}$.
\end{itemize}
\end{proof}
 
The three cases analyzed show that $L(n)$ exhibits a common limiting behavior at the endpoints of the parameter domain: 
both as $n\to 0^+$ and as $n\to\infty$, the length 
converges to the value $8a$. In contrast, the case $n=1$ leads to 
$L(1)=4\sqrt{2}\,a$, a value that admits a direct 
geometric interpretation.

\begin{corollary}
The arc length of the supercircle attains its geometric lower bound at $n=1$, and this length is given by
\[
L(1)=4\sqrt{2}\,a.
\]
\end{corollary}

\begin{proof}
In the first quadrant, all curves of the family connect 
the same points $(r,\theta)=(a,0)$ and 
$(r,\theta)=(a,\pi/2)$. Since the shortest distance between 
two points in the plane is the rectilinear segment joining them, the 
smallest possible length for the arc in the first quadrant is 
$a\sqrt{2}$. For $n=1$, this arc coincides precisely with 
that segment; by symmetry, the total length is four times 
that length, as obtained in Eq.~\eqref{Eq56}. 
Therefore, $L(1)=4\sqrt{2}\,a$ constitutes the geometric lower bound 
of the arc length within the family.
 
This result also admits an analytical confirmation: 
the derivative of $L(n)$ with respect to $n$ vanishes at $n=1$. By differentiating term by term and interchanging the derivative with the integral in Eq.~\eqref{Eq22}, with $u_m$ defined in Eq.~\eqref{Eq23}, for the regime $0<n\leq1$ we obtain
\begin{align}
   \left.\frac{\partial L}{\partial n}\right|_{n=1^-}
   &= 8a\sum_{m=0}^\infty \binom{1/2}{m}
   \int_0^1 \left.\frac{\partial}{\partial n}\left[\frac{1}{n}
   \xi^{u_m-1}(1+\xi)^{-(1+1/n)}\right]\right|_{n=1^-}\,d\xi\nonumber\\
   &= 8a\sum_{m=0}^\infty \binom{1/2}{m}
   \left(2\, m\ln 2-\frac{1}{2}\ln 2\right) = 0,\label{Eq57}
\end{align}
 whereas, for the regime $n\geq1$, 
\begin{align}
   \left.\frac{\partial L}{\partial n}\right|_{n=1^+}
   &= 8a\sum_{m=0}^\infty \binom{1/2}{m}
   \int_0^1 \left.\frac{\partial}{\partial n}\left[\frac{1}{n}
   \xi^{u_m-1}(1+\xi)^{-(1+1/n)}\right]\right|_{n=1^+}\,d\xi\nonumber\\
   &= 8a\sum_{m=0}^\infty \binom{1/2}{m}
   \left(-2\, m\ln 2+\frac{1}{2}\ln 2\right) = 0.\label{Eq58}
\end{align}
 
The vanishing of both derivatives confirms that $n=1$ is a stationary point of $L(n)$ and, together with the preceding geometric argument, allows us to identify it as a global minimum of the arc length in the family of supercircles.
\end{proof}

\subsection{The Circle and a Hypergeometric Representation of $\mathbf{\pi}$}

\begin{corollary}
By specializing the hypergeometric representation of the arc length to the circular case $n=2$, we obtain the following representation of the constant $\pi$:
\begin{equation}\label{Eq60}
    \pi = \sum_{m=0}^\infty \frac{4}{1+2m}\binom{1/2}{m}
    \,{}_2F_1\!\left(\frac{3}{2},\,\frac{1+2m}{2};\,
    \frac{3+2m}{2};\,-1\right).
\end{equation}
\end{corollary}

\begin{proof}
Among the curved cases of the family, the circle ($n=2$) occupies a distinguished place because of its fundamental geometric role. Substituting $n=2$ into Eq.~\eqref{Eq26} for the regime $n\geq 1$, we obtain $u_m=(1+2m)/2$, so that
\begin{equation}\label{Eq59}
    L_{\circ} = 8a\sum_{m=0}^\infty \frac{1}{1+2m}\binom{1/2}{m}
    \,{}_2F_1\!\left(\frac{3}{2},\,\frac{1+2m}{2};\,
    \frac{3+2m}{2};\,-1\right).
\end{equation}
Since the diameter of the circle is $2a$, the ratio between its length and its diameter is given by $L_{\circ}/(2a)$. By definition, this ratio corresponds to $\pi$; therefore, Eq.~\eqref{Eq60} follows from Eq.~\eqref{Eq59}.
\end{proof}

\begin{remark}
Eq.~\eqref{Eq60} constitutes a hypergeometric representation of $\pi$ obtained as a particular case of the general formulation for the arc length. Its validity follows directly from the convergence of Eq.~\eqref{Eq26} for all $n>0$; therefore, the subsequent numerical study is aimed at examining its efficiency as an approximation formula.
\end{remark}

\section{Numerical Results}

\subsection{Truncation Criterion}

The formulation obtained in Eq.~\eqref{Eq26} expresses the
arc length of the supercircle by means of an infinite series of
hypergeometric functions. Therefore, its numerical evaluation requires approximating this series by successive truncations.
For fixed values of $n$ and $a$, we construct a sequence of
partial sums that progressively converges to the value of $L$.
If $L_{\text{previous}}$ and $L_{\text{current}}$ denote two
consecutive approximations, the approximate relative percentage
error is defined as
\begin{equation}\label{Eq61}
    \varepsilon_a =
    \frac{L_{\text{current}}-L_{\text{previous}}}{L_{\text{current}}}
    \times 100\%.
\end{equation}
The iterative process stops when the absolute value of this
error satisfies the condition
\begin{equation}\label{Eq62}
    |\varepsilon_a| < \varepsilon_s,
\end{equation}
where $\varepsilon_s$ is a prescribed percentage tolerance~\cite{R-17,R-18}.
Since the series~\eqref{Eq26} has alternating convergence,
the criterion~\eqref{Eq62} is interpreted as a measure of proximity between
successive iterations and, therefore, for each value of $n$, it determines the minimum truncation order $m^*$ required to reach the proximity prescribed by $\varepsilon_s$.
 
The truncation order is denoted by
\begin{equation}\label{Eq63}
    m^* = \min\bigl\{m \geq 1 :\,|\varepsilon_a|<\varepsilon_s\bigr\},
\end{equation}
where the condition $m \geq 1$ reflects that the stopping criterion
is activated from the second partial sum onward, once
$L_{\text{previous}}$ is defined.

\subsection{Global Behavior of $\mathbf{L(n)}$}

The iterative procedure produces a single value of $L$ for each fixed value of $n$; therefore, the study of the global behavior
of the arc length requires repeating the computation over
a suitable discretization of the domain of the parameter $n$.
For this purpose, we adopt a logarithmically uniform partition
of the interval
$[n_{\min}, n_{\max}]$, defined by
\begin{equation}\label{Eq64}
    n_i = 10^{\log_{10}(n_{\min})
    + \frac{i-1}{N-1}
    \left[\log_{10}(n_{\max})-\log_{10}(n_{\min})\right]},
    \qquad i=1,2,\ldots,N.
\end{equation}
This choice distributes the evaluation points uniformly on a logarithmic scale between the regimes $0<n<1$ and $n>1$, preventing the sampling from being dominated by the large scales of $n$. For each node $n_i$, we accumulate the terms of
Eq.~\eqref{Eq26} corresponding to $m=0,1,2,\ldots$,
until the criterion~\eqref{Eq62} is satisfied. Fig.~\ref{F2} a) shows the \texttt{Python~3.13}
code of the procedure described above, with the parameters $a=1$,
$\varepsilon_s=10^{-6}$\%, $N=1000$ nodes, and domain
$n\in[0.01,\,1000]$.

\begin{figure}[h]
\centering
\includegraphics[width=1.00\linewidth]{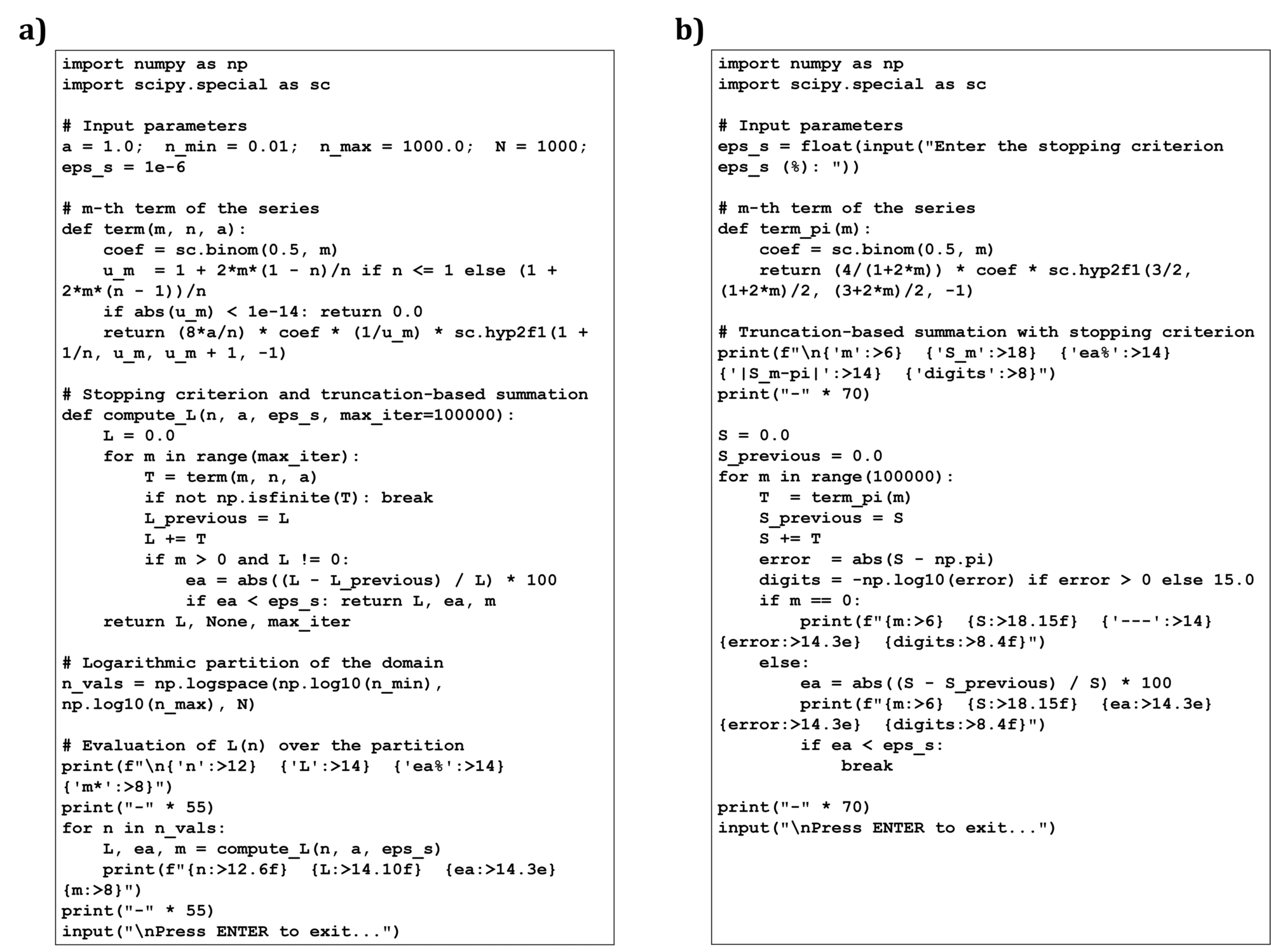}
\caption{Code in \texttt{Python~3.13}: a) iterative numerical 
calculation of the arc length $L(n)$ by truncating 
Eq.~\eqref{Eq26}; b) iterative calculation of the
hypergeometric representation of $\pi$,
Eq.~\eqref{Eq60}.}
\label{F2}
\end{figure}

Figs.~\ref{F3} and~\ref{F4} present the numerical results
for $L(n)$ and $m^*(n)$, respectively. The curve
$L(n)$ is consistent with the asymptotic behaviors
established in Eqs.~\eqref{Eq44} and~\eqref{Eq51},
as well as with the minimum in Eq.~\eqref{Eq56}: the
length converges to $L\to 8a$ at both endpoints of the family
and reaches its minimum value $L=4\sqrt{2}\,a$ at $n=1$, so that
for every $n>0$ we have
\begin{equation}\label{Eq65}
    4\sqrt{2}\,a \leq L(n) < 8a.
\end{equation}
 
\begin{figure}[h!]
\centering
\includegraphics[width=0.9\linewidth]{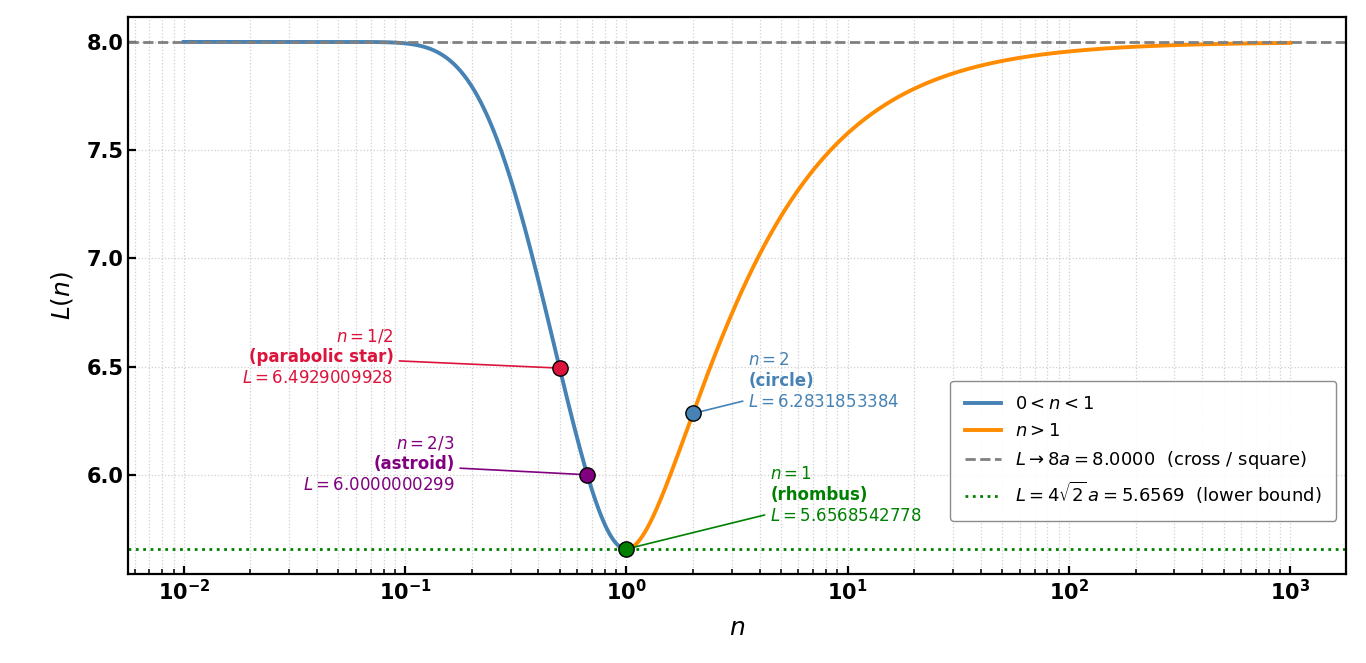}
\caption{Behavior of the arc length $L$ of the 
supercircle as a function of the parameter $n$ on a logarithmic scale.}
\label{F3}
\end{figure}

Fig.~\ref{F4} shows that the truncation order $m^*$
does not exhibit uniform behavior with respect to $n$. We
observe a pronounced increase around $n=1$, whose
maximum value over the logarithmic partition is $m^*=4110$,
reached at the node $n=1.0046$. This localized accumulation
indicates significantly slower convergence of
Eq.~\eqref{Eq26} near the transition value between the regimes $0<n\leq 1$ and $n\geq 1$.
The logarithmic partition considered does not include exactly the node
$n=1$; therefore, the plot does not reflect the corresponding value of
$m^*(n)$. Table~\ref{TI} shows that for the rhombus
($n=1$), we obtain $m^*=73551$, the maximum value of the truncation
order for $\varepsilon_s=10^{-6}$\%.

\begin{figure}[h!]
\centering
\includegraphics[width=0.8\linewidth]{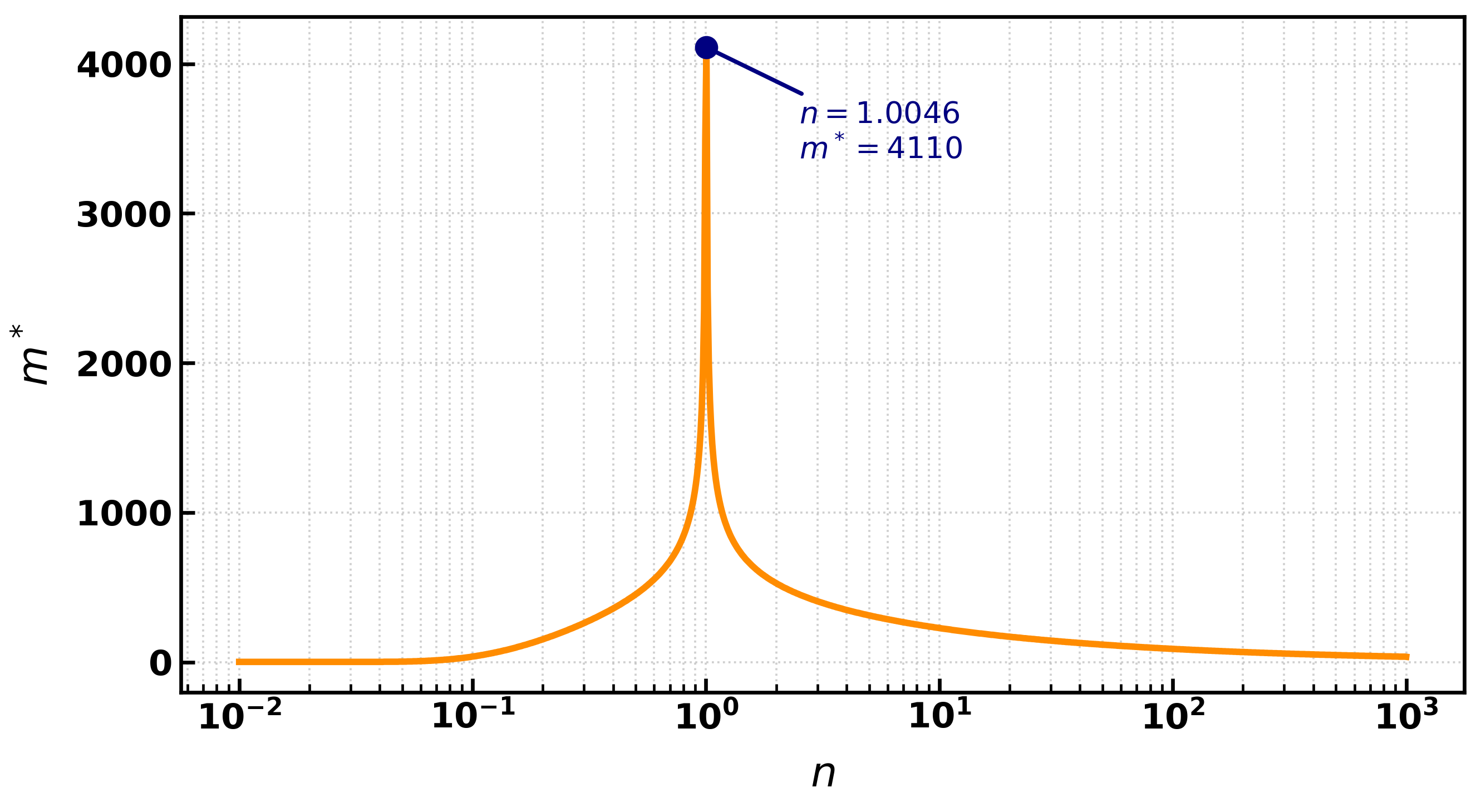}
\caption{Truncation order $m^*$ as a function of the parameter 
$n$ on a logarithmic scale, with $a=1$, $\varepsilon_s=10^{-6}$\%.}
\label{F4}
\end{figure}

\subsection{Validation on Known Analytical Cases}

The markers included in Fig.~\ref{F3} correspond to
pointwise evaluations of four classical supercircles
and are superimposed on the curve obtained from the logarithmic
partition: the parabolic star ($n=1/2$), the astroid
($n=2/3$), the rhombus ($n=1$), and the circle ($n=2$), illustrated
in Fig.~\ref{F1}. These geometries have exact and compact
analytical expressions for the arc length in
terms of the parameter $a$, which makes them natural
validation cases for Eq.~\eqref{Eq26}:
\begin{equation}\label{Eq66}
    L\!\left(\tfrac{1}{2}\right)
    = 4a\!\left[1+\frac{\sqrt{2}}{2}\operatorname{arcsinh}(1)\right],
    \quad
    L\!\left(\tfrac{2}{3}\right) = 6a,
    \quad
    L(1) = 4\sqrt{2}\,a,
    \quad
    L(2) = 2\pi a.
\end{equation}
Table~\ref{TI} compares the analytical values of $L$ with
those obtained by truncating
Eq.~\eqref{Eq26}. The true relative percentage error is defined as
\begin{equation}\label{Eq67}
    \varepsilon_{v\%} = \left|
    \frac{L_{\text{numerical}}-L_{\text{analytical}}}
    {L_{\text{analytical}}}
    \right|\times 100\%.
\end{equation}
 
\begin{table}[h!]
\centering
\caption{Comparison between the analytical and
numerical lengths for four particular supercircles, with $a=1$ and
$\varepsilon_s=10^{-6}$\%. The value of $\varepsilon_{a\%}$
corresponds to the last step satisfying the
stopping criterion~\eqref{Eq62}.}
\begin{tabular}{lcccccc}
\hline\hline
\textbf{Supercircle} & $n$ & $L_{\text{analytical}}$ &
$L_{\text{numerical}}$ & $\varepsilon_{a\%}$ &
$\varepsilon_{v\%}$ & $m^*$ \\
\hline
Parabolic star & $1/2$ & $6.4929009606$ & $6.4929009928$
& $9.96\times10^{-7}$ & $4.97\times10^{-7}$ & $453$ \\
Astroid            & $2/3$ & $6.0000000000$ & $6.0000000299$
& $9.98\times10^{-7}$ & $4.98\times10^{-7}$ & $631$ \\
Rhombus               & $1$   & $5.6568542495$ & $5.6568542778$
& $1.00\times10^{-6}$ & $5.00\times10^{-7}$ & $73551$ \\
Circle             & $2$   & $6.2831853072$ & $6.2831853384$
& $9.97\times10^{-7}$ & $4.97\times10^{-7}$ & $527$ \\
\hline\hline
\end{tabular}
\label{TI}
\end{table}
 
The results in Table~\ref{TI} confirm that
Eq.~\eqref{Eq26} reproduces the
reference analytical lengths with high precision: in all cases,
$\varepsilon_{v\%}$ remains of order $10^{-7}$\%.
The cases $n=1/2$, $n=2/3$, and $n=2$ require truncation orders of the order of a few hundred, whereas the case $n=1$, corresponding to the
rhombus, requires a truncation order of tens of thousands of terms.

Fig.~\ref{F4} shows that the truncation order $m^{*}$ increases sharply when $n$ approaches $1$. This behavior is accentuated in the pointwise evaluation of the rhombus, for which Table~\ref{TI} reports $m^{*}=73551$, the largest value obtained under the stopping criterion considered. This result is consistent with the convergence analysis of the series~\eqref{Eq26}: for $n=1$, we have $u_m=1$, and the general term is controlled by the decay of the generalized binomial coefficient, of order $O(m^{-3/2})$, whereas for $n\neq 1$ the bound $O(m^{-5/2})$ is obtained. Therefore, although the rhombus is the geometrically simplest case of the family and its length is obtained in an elementary way, its representation through the series~\eqref{Eq26} is numerically less efficient. Consequently, the value $n=1$, which minimizes the arc length within the family, also corresponds to the case that requires the largest truncation order in the numerical evaluation.

\subsection{Convergence to $\mathbf{\pi}$}

By specializing the circular case ($n=2$) in
Eq.~\eqref{Eq26} and considering the ratio $L_\circ/(2a)$,
we obtain the hypergeometric representation of
$\pi$ given in Eq.~\eqref{Eq60}. Its evaluation by
successive truncations allows us to study the convergence of
the partial sums $S_m$ and to quantify the precision achieved
as a function of the tolerance $\varepsilon_s$.
 
Fig.~\ref{F2} b) shows the code in \texttt{Python~3.13} corresponding to the
iterative procedure applied to Eq.~\eqref{Eq60}.
The algorithm uses the same stopping criterion
$|\varepsilon_a|<\varepsilon_s$ defined above and records, at each iteration, the partial sum $S_m$,
the absolute error $|S_m-\pi|$, and the number of correct digits
estimated from that error.
 
Fig.~\ref{F6} summarizes the convergence behavior.
In panel~(a), the first $16$ partial sums oscillate
around $\pi$, alternating above and below the
reference value with progressively smaller amplitudes.
This oscillation is due to the alternating character of the
binomial coefficients $\binom{1/2}{m}$ for $m\geq 1$.
In panel~(b), the absolute error $|S_m-\pi|$, represented
on a logarithmic scale, decreases as the number of partial sums considered increases.
 
\begin{figure}[h]
\centering
\includegraphics[width=0.95\linewidth]{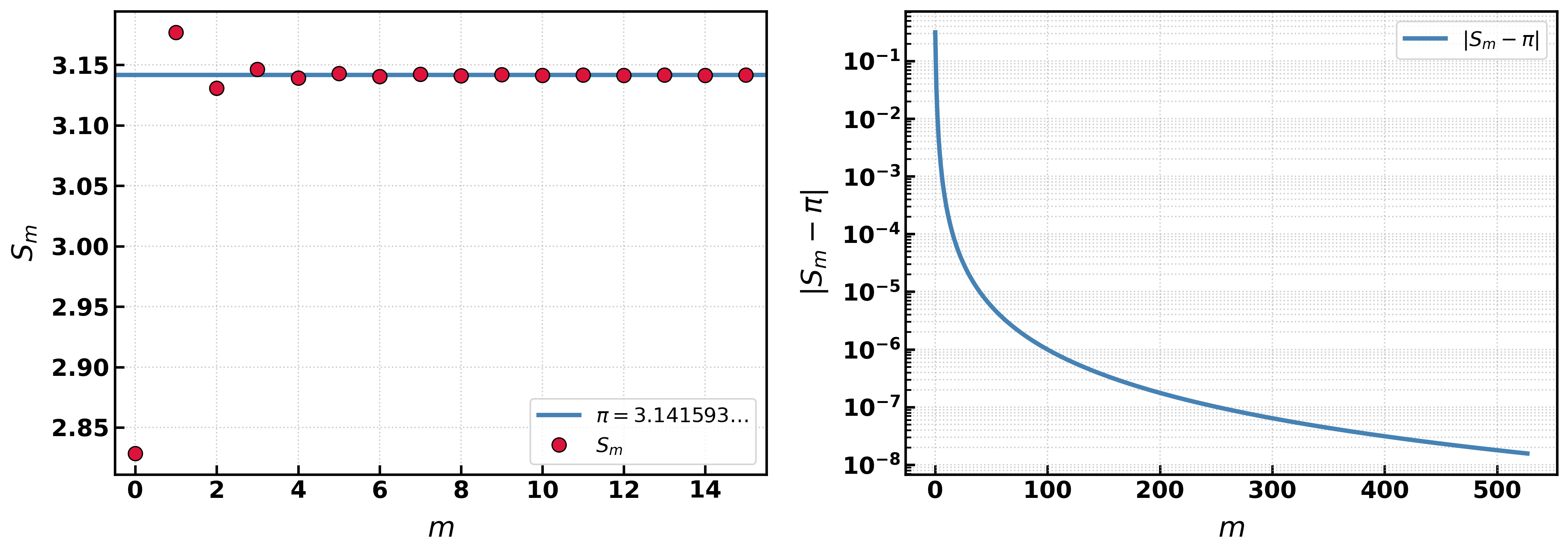}
\caption{Convergence to $\pi$:
(a) first $16$ partial sums $S_m$ and
(b) evolution of the error $|S_m-\pi|$ on a logarithmic scale,
with $\varepsilon_s=10^{-6}$\%.}
\label{F6}
\end{figure}
 
Table~\ref{T2} presents the approximations obtained for
a profile of percentage tolerances $\varepsilon_s=10^{s}$,
with $s=0,-1,-2,\ldots,-10$. The number of estimated correct digits is calculated as
\begin{equation}\label{Eq68}
     -\log_{10}|S_{m^*}-\pi| ,
\end{equation}
where $S_{m^*}$ is the partial sum that satisfies the stopping
criterion for the first time. In the range considered, each
decimal decrease of $\varepsilon_s$ produces approximately
one additional digit; nevertheless, this observation is empirical
and is not intended to characterize the asymptotic convergence rate
of the series.
 
\begin{table}[h!]
\centering
\caption{Approximations of $\pi$ by means of the
hypergeometric formulation~\eqref{Eq60} for a profile of percentage stopping tolerances $\varepsilon_s=10^{s}$, with
$s=0,-1,\ldots,-10$. The number of estimated correct digits is calculated using Eq.~\eqref{Eq68}.}
\begin{tabular}{cccc}
\hline\hline
\textbf{Tolerance} & \textbf{Order} & \textbf{Partial sum} & \textbf{Estimated digits} \\
$\varepsilon_s$ & $m^*$ & $S_{m^*}$ & Eq.~\eqref{Eq68} \\
\hline
$10^{0}$   & $3$     & $3.146280449988269$ & $2.3$  \\
$10^{-1}$  & $6$     & $3.140605171464835$ & $3.0$  \\
$10^{-2}$  & $14$    & $3.141463420680137$ & $3.9$  \\
$10^{-3}$  & $34$    & $3.141578143007277$ & $4.8$  \\
$10^{-4}$  & $84$    & $3.141591123065441$ & $5.8$  \\
$10^{-5}$  & $210$   & $3.141592497994520$ & $6.8$  \\
$10^{-6}$  & $527$   & $3.141592669214308$ & $7.8$  \\
$10^{-7}$  & $1322$  & $3.141592652021008$ & $8.8$  \\
$10^{-8}$  & $3320$  & $3.141592653432782$ & $9.8$  \\
$10^{-9}$  & $8339$  & $3.141592653605508$ & $10.8$ \\
$10^{-10}$ & $20946$ & $3.141592653588247$ & $11.8$ \\
\hline\hline
\end{tabular}
\label{T2}
\end{table}
 
The results in Table~\ref{T2} show that the
representation~\eqref{Eq60} requires substantial growth
of $m^*$ to gain additional digits of $\pi$.
In particular, for $\varepsilon_s=10^{-6}$\% we obtain
$m^*=527$, a value consistent with that reported in
Table~\ref{TI} for the arc length of the circle, and for
$\varepsilon_s=10^{-10}$\% approximately $11$ correct digits of $\pi$ are reached with $m^*=20946$. Therefore, the series is not a competitive tool
for the efficient computation of digits of $\pi$. Its
main interest is analytical and geometric: the constant
$\pi$ emerges by specializing the
general result~\eqref{Eq26} to the circular case of the family,
corresponding to $n=2$.

\section{Conclusions}
We obtain an infinite-series representation for the arc length of a supercircle, expressed in terms of generalized binomial coefficients and Gauss hypergeometric functions. The resulting formulation depends on the scale parameter $a$ and the shape parameter $n$, and is organized into two regimes determined by the value of $n$. The series converges for all $n>0$, although its convergence rate is not uniform: in the transition regime $n=1$ the convergence is slower, whereas it improves progressively as $n$ moves away from unity.

We establish the consistency of the formulation by analyzing the limiting cases $n\to 0^+$ and $n\to\infty$, for which the length converges to the value $8a$, as well as the case $n=1$, where we recover the length $4\sqrt{2}\,a$ corresponding to the rhombus. This latter value constitutes the geometric lower bound of the family. The numerical verification against configurations with exact arc length ---the parabolic star, the astroid, the rhombus, and the circle--- shows that the series reproduces the analytical values with high precision.
As an additional result, specializing the general formula to the circular case $n=2$ leads to a hypergeometric representation of $\pi$. When the prescribed tolerance is reduced by one order of magnitude, approximately one additional correct digit is obtained in its numerical approximation. This trend indicates that reaching higher precision requires considering increasing values of the truncation order $m^*$. Therefore, although this representation is not proposed as a competitive formula for the efficient computation of digits of $\pi$, it has analytical and geometric interest, since the constant emerges naturally as a consequence of the circular case of the family of supercircles.

In future work, it would be natural to study the acceleration of the convergence of the series by means of procedures such as Aitken's $\Delta^2$ algorithm, as well as to explore parametrizations of the formulation obtained for families of constant area or constant length. Likewise, the approach developed here suggests possible extensions to more general superellipses and to applications in the geometric modeling of problems in science and engineering.

\appendix
\section*{\centering{Appendices}}
\numberwithin{equation}{section}

\section{Asymptotic Expansion of the Binomial Coefficient for $\mathbf{\alpha=1/2}$}\label{AA}

The generalized binomial expansion is defined as~\cite{R-19}
\begin{equation}\label{A.1}
    (1+x)^\alpha
    = \sum_{m=0}^{\infty} \binom{\alpha}{m}\, x^m
    = \sum_{m=0}^{\infty} \frac{(-1)^m(-\alpha)_m}{m!}\,x^m, \qquad |x| < 1, \quad \alpha \in \mathbb{R},
\end{equation}
where $(-\alpha)_m$ denotes the Pochhammer symbol. The series converges for $|x|<1$ and also at the endpoints $x=\pm 1$ when $\alpha>0$. In particular, for $\alpha=1/2$, the binomial coefficient takes the form
\begin{equation}\label{A.2}
    \binom{1/2}{m}
    = \frac{(-1)^m\left(-\tfrac{1}{2}\right)_m}{m!}.
\end{equation}
From the identities $(v)_m = \Gamma(v+m)/\Gamma(v)$, with $v=-1/2$, $m!=\Gamma(m+1)$ and $\Gamma(-1/2)=-2\sqrt{\pi}$, we obtain
\begin{equation}\label{A.3}
    \binom{1/2}{m} = (-1)^{m+1}\,
    \frac{\Gamma\!\left(m-\frac{1}{2}\right)}{2\sqrt{\pi}\,\Gamma(m+1)}.
\end{equation}
Therefore, by taking absolute values,
\begin{equation}\label{A.4}
    \left|\binom{1/2}{m}\right| = \frac{\Gamma\!\left(m-\frac{1}{2}\right)}{2\sqrt{\pi}\,\Gamma(m+1)}.
\end{equation}
By applying the asymptotic behavior, for sufficiently large $m$, of the ratio of Gamma functions~\cite{R-16,R-20},
\begin{equation}\label{A.5}
    \frac{\Gamma(m+v)}{\Gamma(m+w)} \sim m^{v-w},
\end{equation}
with $v=-1/2$ and $w=1$, we obtain
\begin{equation}\label{A.6}
    \left|\binom{1/2}{m}\right| \simeq \frac{1}{2\sqrt{\pi}}\,m^{-3/2},
\end{equation}
that is,
\begin{equation}\label{A.7}
    \left|\binom{1/2}{m}\right| \sim O\!\left(m^{-3/2}\right).
\end{equation}
Consequently, the absolute value of the generalized binomial coefficient for $\alpha=1/2$ is asymptotically of order $m^{-3/2}$.

\section{Evaluation of the Series $\mathbf{S_1}$}\label{BB}

We consider the series
\begin{equation}\label{B.1}
    S_1 = \sum_{m=1}^\infty \frac{1}{2m}\binom{1/2}{m}
    \,{}_2F_1\!\left(1,\,2m;\,2m+1;\,-1\right).
\end{equation}
To evaluate this sum, we first transform the
hypergeometric function that appears in equation~\eqref{B.1}.
Applying Euler's transformation
\begin{equation*}
    {}_2F_1(v,u;w;\gamma)
    = (1-\gamma)^{-v}
      {}_2F_1\!\left(v,w-u;w;\frac{\gamma}{\gamma-1}\right),
\end{equation*}
with $v=1$, $u=2m$, $w=2m+1$ and $\gamma=-1$, we obtain
\begin{equation}\label{B.2}
    {}_2F_1(1,2m;2m+1;-1)
    = \frac{1}{2}\,{}_2F_1\!\left(1,1;\,2m+1;\,\frac{1}{2}\right).
\end{equation}
Then, by means of Euler's integral representation,
equation~\eqref{Eq24}, with $v=1$, $u=1$, $w=2m+1$
and $\gamma=1/2$, and using $\Gamma(2m+1)/\Gamma(2m)=2m$,
we obtain
\begin{equation}\label{B.3}
    \frac{1}{2}\,{}_2F_1\!\left(1,1;\,2m+1;\,\frac{1}{2}\right)
    = m\int_0^1\frac{(1-\xi)^{2m-1}}{1-\xi/2}\,d\xi.
\end{equation}
The change of variable $t=1-\xi$ transforms the denominator
according to $1-\xi/2=(1+t)/2$, and the integration limits
again remain between $0$ and $1$. Therefore, equations~\eqref{B.2}
and~\eqref{B.3} give
\begin{equation}\label{B.4}
    {}_2F_1(1,2m;2m+1;-1) = 2m\int_0^1\frac{t^{2m-1}}{1+t}\,dt.
\end{equation}
Substituting equation~\eqref{B.4} into~\eqref{B.1},
the factor $1/(2m)$ cancels with $2m$, and we obtain
\begin{equation}\label{B.5}
    S_1 = \sum_{m=1}^\infty\binom{1/2}{m}
          \int_0^1\frac{t^{2m-1}}{1+t}\,dt.
\end{equation}
By interchanging the sum and the integral, equation~\eqref{B.5} is rewritten as
\begin{equation}\label{B.6}
    S_1 = \int_0^1\frac{1}{t(1+t)}
          \sum_{m=1}^\infty\binom{1/2}{m}t^{2m}\,dt.
\end{equation}
The inner sum is identified with the generalized binomial series
\begin{equation}\label{B.7}
(1+t^2)^{1/2}=\sum_{m=0}^\infty\binom{1/2}{m}t^{2m},
\end{equation}
which converges since $\alpha=1/2>-1$ and $t^2\leq 1$
for all $t\in[0,1]$. Separating the term $m=0$, we have
\begin{equation}\label{B.8}
    \sum_{m=1}^\infty\binom{1/2}{m}t^{2m} = \sqrt{1+t^2}-1.
\end{equation}
Substituting equation~\eqref{B.8} into~\eqref{B.6},
we obtain the integral representation
\begin{equation}\label{B.9}
    S_1 = \int_0^1\frac{\sqrt{1+t^2}-1}{t(1+t)}\,dt.
\end{equation}
This last integral converges, and its evaluation leads to
\begin{equation}\label{B.10}
   S_1 = 2\ln 2 - \sqrt{2}\,\ln(1+\sqrt{2}).
\end{equation}

\section{Elementary Reduction of a Particular Case of $\mathbf{{}_2F_1}$}\label{CC}

We consider the following particular value of the hypergeometric function:
\begin{equation}\label{C.1}
   {}_2F_1\!\left(1+\frac{1}{n},\,1;\,2;\,-1\right).
\end{equation}
By means of Euler's integral representation,
equation~\eqref{Eq24}, with $v=1+1/n$, $u=1$, $w=2$ and $\gamma=-1$,
and since $\Gamma(1)=\Gamma(2)=1$, we obtain
\begin{equation}\label{C.2}
    {}_2F_1\!\left(1+\frac{1}{n},1;\,2;\,-1\right)
    = \int_0^1(1+\xi)^{-(1+1/n)}\,d\xi.
\end{equation}
Since $1+\xi\geq 1>0$ for all $\xi\in[0,1]$, the integrand
is continuous and bounded on this interval; therefore, the definite integral
is evaluated directly by the power rule:
\begin{equation}\label{C.3}
    \int_0^1(1+\xi)^{-(1+1/n)}\,d\xi
= n\!\left(1-2^{-1/n}\right).
\end{equation}
Therefore,
\begin{equation}\label{C.4}
    {}_2F_1\!\left(1+\frac{1}{n},1;\,2;\,-1\right)
    = n\!\left(1-2^{-1/n}\right).
\end{equation}

\end{document}